\documentclass[12pt,letterpaper]{amsart}
\usepackage{geometry,mathpazo}
\usepackage{mathrsfs,amsmath,amssymb,amsthm,amsfonts}
\newtheorem{theorem}{Theorem}
\theoremstyle{definition}
\newtheorem{example}{Example}
\theoremstyle{remark}
\newtheorem{remark}{Remark}
\newcommand{\C}{\mathbb{C}}
\newcommand{\D}{\Omega}
\newcommand{\Dc}{\overline{\Omega}}
\newcommand{\zb}{\overline{z}}
\newcommand{\dbar}{\overline{\partial}}
\newcommand{\ep}{\varepsilon}
\usepackage{hyperref}
\begin{document}

\title[Compactness of the $\overline{\partial}$-Neumann problem]{A potential
theoretic characterization of compactness of the $\overline{\partial}$-Neumann
problem}
\author{S\"{o}nmez \c{S}ahuto\u{g}lu}
\address{University of Toledo, Department of Mathematics, Toledo, OH 43606, USA}
\email{sonmez.sahutoglu@utoledo.edu}

\subjclass[2010]{32W05}
\date{}
\keywords{Potential theory, $\dbar$-Neumann problem}
\thanks{This article is based on a part of the author’s Ph.D. thesis
\cite{SahutogluThesis}.}
\thanks{The author is supported in part by University of Toledo's Summer
Research Awards and Fellowships Program.}

\begin{abstract}
We give a potential theoretic characterization for compactness of the 
$\dbar$-Neumann problem on smooth bounded pseudoconvex domains in $\C^n.$
\end{abstract}

\maketitle

Let $\D$ be a domain in $\C^n$ with $C^{\infty}$-smooth boundary. The domain
$\D$ is said to be pseudoconvex if the Levi form of $\D$, the restriction of the
complex Hessian of a defining function onto complex tangent space, is positive
semi-definite on the boundary, $b\D,$ of $\D$. On bounded pseudoconvex domains,
H\"ormander \cite{Hormander65} showed that the $\dbar$-Neumann operator on $\D,$
the solution operator for $\Box$ is a bounded operator on $L_{(0,1)}^{2}(\D)$
(here $\Box =\dbar\dbar^*+\dbar^*\dbar$ and $\dbar^*$ is the Hilbert space
adjoint of $\dbar$).  We refer the reader to \cite{ChenShawBook,StraubeBook} for
more information about the $\dbar$-Neumann problem.

Compactness of the $\dbar$-Neumann operator is important to study as it is
weaker than global regularity \cite{KohnNirenberg65} and it interacts with the
boundary geometry. For example, even though the general case is still open, in
same cases it is know that existence of an analytic disc in the boundary is an
obstruction for compactness of the $\dbar$-Neumann problem (see, for example,
\cite{FuStraube98,FuStraube01,SahutogluThesis, SahutogluStraube06,StraubeBook}).
Recently, Straube and  Munasinghe \cite{MunasingheStraube07} (see also
\cite{MunasingheThesis}) studied compactness using geometric conditions
involving short time flows on the boundary (this was done in $\C^2$ earlier by
Straube \cite{Straube04}). \c{C}elik and Straube \cite{CelikStraube09} (see also
\cite{CelikThesis}) explored compactness in relation to the so called
``compactness multipliers''.

Compactness of the $\dbar$-Neumann problem has been studied using some
potential theoretic conditions  by Catlin \cite{Catlin84} using property 
$(P)$ and later by McNeal \cite{McNeal02} using property $(\widetilde{P}).$ In
this paper we would like to give a new potential theoretic characterization for
compactness of the $\dbar$-Neumann problem. We refer the reader to
\cite[Proposition 4.2]{StraubeBook} for other equivalent conditions. We would
like to note that a similar characterization has been done by Haslinger in
\cite{Haslinger}. 

Let $\D$ be a smooth bounded domain in $\C^n,K\subset b\D,$ and $U$ be an open
neighborhood of $K.$ We denote the $L^2$ norm and  Sobolev $-1$ norm of a
function $f\in L^2(\D)$ by $\|f\|$ and $\|f\|_{-1},$ respectively. Let
$I=\{i_1,i_2,\ldots,i_p\}\subset \mathbb{N}$ such that $j_1<i_2<\cdots<i_p.$
Then we use the notation $dz_I=dz_{i_1}\wedge dz_{i_2}\wedge \cdots \wedge 
dz_{i_p}$ and  $d\zb_I=d\zb_{i_1}\wedge d\zb_{i_2}\wedge \cdots \wedge 
d\zb_{i_p}.$ Define 
\[ C_{0,(p,q)}^{\infty}(U)=\left \{\sum_{\scriptscriptstyle |I|=p,
|J|=q}f_{IJ}dz_I\wedge d\bar z_J: f_{IJ}\in C^{\infty}_0(U)\right\}\]
for $0\leq q\leq n.$ Define $\lambda_{(p,q)}(U)$ as follows: for $0\leq p\leq
n$ and $1\leq q\leq n-1$  let us define 
\[\lambda_{(p,q)}(U)=\inf \left\{ \frac{\|\overline{\partial}f\|^2+\| 
\overline{\partial}^*f\|^2}{\|f\|^2}: f\in Dom(\overline{\partial}^*)\cap
C_{0,(p,q)}^{\infty}(U),f\not\equiv 0 \right\}\]
and
\[ \lambda_{(p,0)}(U) =\inf \left\{ \frac{\|\overline{\partial}f\|^2}{\|f\|^2}:
f\in (Ker\overline{\partial})^{\bot}\cap C_{0,(p,0)}^{\infty}(U),f\not\equiv
0 \right\}  \]  
where $(Ker\overline{\partial})^{\bot}$ is the orthogonal complement of
$(Ker\overline{\partial})$ in $L_{(p,0)}^2(\D)$ (square integrable
$(p,0)$-forms on $\D$). Notice that $\lambda_{(p,q)}(U)\leq \lambda_{(p,q)}(V)$
if $ V\subset U.$ In this paper a finite type is meant in the sense of
D'Angelo \cite{D'Angelo82}  and infinite type point means a point that is not
finite type.  

\begin{theorem}\label{thm1}
Let $\D$ be a smooth bounded pseudoconvex domain in $\C^n,n\geq 2$ and 
$0\leq p\leq n,0\leq q \leq n-1,$ be given. Then the following are equivalent:
\begin{itemize}
\item[(i)] the $\overline{\partial}$-Neumann operator $N_{(p,q)}$ of $\D$ is
compact on $L_{(p,q)}^2(\D),$
\item[(ii)] for any compact set $K \subset b\D$ and $M>0$ there exists an open
neighborhood $U$ of $K$ such that $\lambda_{(p,q)}(U)>M,$
\item[(iii)] for any  $M>0$ there exists an open neighborhood $U$ of the set of 
infinite type points in $b\D$ such that
$\lambda_{(p,q)}(U)>M.$
\end{itemize}
\end{theorem}

\begin{remark} The definition of $\lambda_{(p,q)}$ is closely connected  to 
the so-called compactness  estimates (see \eqref{comp-est} in the proof of
Theorem \ref{thm1}) as well as 
Morey-Kohn-H\"{o}rmander formula (see, for example, 
\cite[Proposition 4.3.1]{ChenShawBook} or \cite[Proposition
2.4]{StraubeBook}) and property $(P)$ of Catlin. 

One can show that the Morey-Kohn-H\"{o}rmander formula  implies that for  a
smooth bounded pseudoconvex $\D\subset \C^n,$ a non-positive function $b\in
C^2(\Dc),$  and $u\in Dom(\dbar)\cap Dom(\dbar^*)\cap C^1_{(p,q)}(\Dc)$ we have 
\[\sum_{J,K}'\sum_{j,k=1}^n\int_{\D}e^b\frac{\partial^2 b}{\partial z_j\partial
\overline{z}_k}u_{J,jK}\overline{u_{J,kK}}dV\leq \|\dbar u\|^2+\|\dbar^*u\|^2.\]
where $ u=\sum'_{J,K}\sum_{k=1}^n u_{J,kK}dz^J\wedge
d\overline{z_k}\wedge d\overline{z_K}$ and the prime indicates that the sum is
taken over strictly increasing $(p,q-1)$-tuples $(J,K).$ If the domain $\D$
satisfies property $(P)$ then one can choose $b$ to be bounded from below by
$-1$ and with arbitrarily large complex Hessian on the boundary of $\D.$ Then on
a small neighborhood on the boundary the Hessian is still large. Hence
$\lambda_{(p,q)}(U)$ will be arbitrarily large for a sufficiently small
neighborhood $U$ of $b\D.$ 
\end{remark}

We would like to give a simple example below to show that one can use this
characterization to show that, in some cases, compactness of the $\dbar$-Neumann
problem excludes analytic disks from the boundary. We do not claim any 
originality in this example as it is a special case of Catlin's result  
\cite[Proposition 1]{FuStraube01}.     

\begin{example} 
Let $\D$ be a smooth bounded pseudoconvex domain in $\C^2$ such that
$\D\subset \{z\in \C^2:\text{Im}(z_2)<0\}$ and 
$\{z\in \C^2:\text{Im}(z_2)=0,|z_1|^2+|z_2|^2<1\}\subset b\D.$ 

\noindent \textit{Claim:} The $\dbar$-Neumann operator on $\D$ is not compact. 

\noindent \textit{Proof of the Claim:} There
exist positive numbers $a_1<a_2$ such that 
\[D_1\times W_1 \subset \D\cap \{z\in \C^2:|z_1|^2+|z_2|^2<1\}\subset D_2\times
W_2\] where 
$D_1=\{z\in \C:|z|<2/3\},D_2=\{z\in \C:|z|<2\},$ and  
\begin{align*}
 W_1&=\{z=re^{i\theta}\in \C:0<r<a_1, -2\pi/3<\theta<-\pi/3\},\\ 
W_2&=\{z=re^{i\theta}\in \C:0<r<a_2, -4\pi/3<\theta<\pi/3\}.
\end{align*}
 Let $\phi_j(z_1,z_2)=f(z_1)g_j(z_2)d\zb_1$ where $f \in C^{\infty}_0(D_1)$ and
$f\not\equiv 0.$ Later on we will choose 
$g_j\in C^{\infty}_0(\{z\in \C: |z|<j^{-2}\})$  so that 
$\phi_j \in Dom(\dbar)\cap Dom(\dbar^*).$ There exists $a_3>0$ such that 
$D_1\times W\subset \D $, where $W=\{z\in \C:\text{Im}(z)<0, |z|<a_3\},$ and
$\phi_j(z_1,z_2)=0$ for $z\in \Dc\setminus D_1\times W$ and  $j^{-2}<a_3.$ 
Then for $j^{-2}<a_3$ we have 
\begin{align*}
\frac{\|\dbar \phi_j\|^2+\|\dbar^* \phi_j \|^2}{\|\phi_j\|^2} &= 
\frac{ \left\|
\frac{\partial g_{j}(z_2)}{\partial
\zb_2}f(z_1)\right\|^2+
\left\|g_j(z_2) \frac{\partial f(z_1)}{\partial z_1
}\right\|^2}{\|g_j(z_2)f(z_1)\|^2} \\
&\leq \frac{\left\|
\frac{\partial g_{j}(z_2)}{\partial
\zb_2}\right\|^2_{L^2(W_2)}
\|f\|_{L^2(D_2)}}{\|g_j(z_2)\|^2_{L^2(W_1)}\|f\|_{L^2(D_1)}} +
\frac{\left\|
\frac{\partial f(z_1)}{\partial z_1
}\right\|^2_{L^2(D_1)}\|g_j(z_2)\|^2_{L^2(W)}
}{\|f(z_1)\|^2_{L^2(D_1)}\|g_j(z_2)\|^2_{L^2(W)}} \\
&\leq \frac{\left\|
\frac{\partial g_{j}(z_2)}{\partial
\zb_2}\right\|^2_{L^2(W_2)}}{\|g_j(z_2)\|^2_{L^2(W_1)}} +
\frac{\left\|
\frac{\partial f(z_1)}{\partial z_1
}\right\|^2_{L^2(D_1)}}{\|f(z_1)\|^2_{L^2(D_1)}} 
\end{align*}
Let us choose real valued non-negative functions 
$\chi_j \in C^{\infty}_0(-j^{-2},j^{-2})$ such that $\chi_j(-t)=\chi_j(t)$ and
$\chi(t) =1$ for $|t|\leq \frac{1}{4j^2}.$ Since $z^{-2}$ is not integrable on
$W_1\cap B(0,\ep)$ for any $\ep>0,$ we can choose a positive real number
$\alpha_j$ so that 
\[\int_{W_2\cap B(0,1/j)}
\frac{\left|\chi'_j\left(|z_2|^2\right)\right|^2}{|z_2-i\alpha_j|^2}
dV(z_2) \leq \int_{W_1\cap B(0,1/j)}
\frac{\left|\chi_j\left(|z_2|^2\right)\right|^2}{|z_2-i\alpha_j|^2} dV(z_2).\]
Now we define  
$g_j(z_2)=\chi_j\left(|z_2|^2\right) \tau_j(z_2)(z_2-i\alpha_j)^{-1}$
where $\tau_j\in C^{\infty}(\C)$ such that  $\tau_j(z)\equiv 1 $ for
$\text{Im}(z)\leq 0$ and $\tau_j(z)\equiv 0 $ for $\text{Im}(z)\geq \alpha_j/2.$
 Then   we have 
$\phi_j\in C^{\infty}_{0,(0,1)}(U_j)\cap Dom(\dbar^*)$ where 
$U_j= \{z\in \C:|z|<2^{-1}+j^{-1}\}\times \{z\in \C:|z|<j^{-2}\}$ 
and  
\[\left\|\frac{\partial g_j}{\partial \zb_2}\right\|_{L^2(W_2\cap B(0,1/j))}
\leq \|g_j \|_{L^2(W_1\cap B(0,1/j))}.\]
Hence, we constructed a sequence of $(0,1)$-forms $\{\phi_j\}$ such
that $\phi_j\in C^{\infty}_{0,(0,1)}(U_j)\cap Dom(\dbar^*)$ where 
\[K=\{z\in \C^2:|z_1|\leq 1/2, z_2=0\}=\bigcap_{j=1}^{\infty} U_j\subset b\D\] 
and $\frac{\|\dbar\phi_j\|^2+\|\dbar^* \phi_j \|^2}{\|\phi_j\|^2} $ stays 
bounded as $j\to \infty.$ Hence, by Theorem \ref{thm1}, the $\dbar$-Neumann
operator on $\D$ is not compact.
\end{example}

\section*{Proof of Theorem \ref{thm1}}

\begin{proof}[Proof of Theorem \ref{thm1}] We will show the equivalences for
$0\leq p\leq n$ and $1\leq q \leq n-1.$ The proof can be mimicked for the case
$q=0$ using the following: compactness of $N_0$ is equivalent to the following
compactness estimate: for all $\varepsilon >0$ there exists $D_{\varepsilon}>0$
such that 
\[\|g\|^2\leq \varepsilon
\|\overline{\partial}g\|^2+D_{\varepsilon}\|g\|_{-1}^2 \text{ for } g\in
(Ker\overline{\partial})^{\bot}\cap Dom(\overline{\partial})\] 
 
First let us prove that $(i)$ implies $(ii).$ Assume that the
$\overline{\partial}$-Neumann operator of $\D$ is compact, and there exist $K
\subset b\D$ and $M>0$ such that $\lambda_{(p,q)}(U)<M$ for all open
neighborhoods $U$ of $K.$ We may assume that there exist  sequences  of open
neighborhoods $\{U_k\}$ of $K$ and nonzero $(p,q)$-forms $\{f_k\}$ such that 
\begin{itemize}
\item[i.] $U_{k+1}\Subset U_k, K\subset \bigcap_{k=1}^{\infty}U_k\subset
b\D,f_k\in Dom(\overline{\partial}^* )\cap C_{0,(p,q)}^{\infty}(U_k),$
\item[ii.] $\|f_k\|^2=1,\text{ and } \|\overline{\partial}f_k\|^2+\| 
\overline{\partial}^*f_k\|^2<M \text{ for } k=1,2,3,\cdots $
\end{itemize}
Since  $K\subset \bigcap_{k=1}^{\infty}U_k \subset b\D$ (hence $K$ has measure 0
in $\C^n$) and $f_k\in C_{0,(p,q)}^{\infty}(U_k),$ by passing to a subsequence
if necessary, we may assume that $\|f_k-f_l\|^2\geq 1/2.$ 
Compactness of the $\overline{\partial}$-Neumann operator is equivalent to the
following so called compactness estimate (see 
\cite[Proposition 4.2]{StraubeBook} or \cite[Lemma 1]{FuStraube01}): for all
$\varepsilon >0$ there exists $D_{\varepsilon}>0$ such that
\begin{equation}\label{comp-est}\|g\|^2\leq \varepsilon
(\|\overline{\partial}g\|^2+\|\overline{\partial}^*g\|^2)+D_{\varepsilon}\|g\|_{
-1}^2 \text{ for } g\in Dom(\overline{\partial}^*)\cap
Dom(\overline{\partial})\end{equation}
 Choose $\varepsilon=\frac{1}{16M}.$ Since $Dom(\overline{\partial}^* )\cap
C_{0,(p,q)}^{\infty}(U_k) \subset Dom(\overline{\partial}^*)\cap
Dom(\overline{\partial}) $ using \eqref{comp-est} and ii. above we get
\begin{equation}\label{eq1}
\|f_k-f_l\|^2_{-1}\geq \frac{1}{4D_{\varepsilon}}>0 \text{ for } k\neq l 
\end{equation} 
The imbedding from $L^2(\D)$ to $W^{-1}(\D)$ is compact and $\{f_k\}$ is a
bounded sequence in $L^2_{(p,q)}(D).$ Hence $\{f_k\}$ has a convergent
subsequence in $W^{-1}_{(p,q)}(\D)$. This contradicts with \eqref{eq1}. 

 $(ii)$ obviously implies $(iii)$ so we will skip this part. 

Next let us prove that  $(iii)$ implies $(i).$ Let $K$ be the set of infinite
type points in $b\D$ and  
$u\in Dom(\overline{\partial}^*)\cap C_{(p,q)}^{\infty}(\overline{\D})$. Assume
that $\lambda_{(p,q)}(U_k)>k$ where $\{U_k\}$ is a sequence of open
neighborhoods of $K$ such that $U_{k+1}\Subset U_k$ and $ K\subset
\bigcap_{k=1}^{\infty}U_k\subset b\D.$ Let $\varphi_k\in C^{\infty}_0(U_k)$ such
that $0\leq  \varphi_k \leq 1$ and $\varphi_k \equiv 1$ in a neighborhood of
$K.$ Define $\psi_k=1-\varphi_k .$  Notice that $\psi_k$ is supported away from
$K.$ In following estimates, $C_k$ and $C_{k,\ep}$ are general constants meaning
that  the constants depend on the subscripts only but they might change at each
step. Away from $K$ we have subelliptic estimates as $b\D \setminus K$ is the
set of finite type points (see \cite{Catlin87}). Hence, there exists $s>0$ for
all $\varepsilon>0$
there exists $D_{\varepsilon}>0$ such that 
\begin{align} 
\| \psi_k u \|^2  &\leq  \varepsilon \| \psi_k u\|^2_s+D_{\varepsilon}\|\psi_k
u\|^2_{-1} \nonumber \\ 
&\leq \varepsilon C_k(\| \overline{\partial}(\psi_k
u)\|^2+\|\overline{\partial}^*(\psi_k
u)\|^2)+C_{k,\varepsilon}\|u\|^2_{-1}\label{ft-est} \\
&\leq  \varepsilon C_k(\| \overline{\partial}u \|^2+\|\overline{\partial}^*
u\|^2+\|u\|^2)+C_{k,\varepsilon}\|u\|^2_{-1}\nonumber
\end{align}
The first inequality follows because $L^2$ imbedds compactly into $W^s$ for
$s>0.$ We used the compactness estimate for the second inequality. If we use
$\lambda_{(p,q)}(U_k)>k$ we get:
\begin{align} 
\|\varphi_ku\|^2 &\leq \frac{1}{k}(\| \overline{\partial}(\varphi_k
u)\|^2+\|\overline{\partial}^*(\varphi_ku)\|^2)\nonumber \\
&\leq \frac{1}{k}(\| \overline{\partial} u \|^2+\|\overline{\partial}^*u
\|^2)+D_k\|\phi_ku\|^2 \label{eqn2}
\end{align}
where $\phi_k\equiv 0$ in a neighborhood of $K, D_k>0,$ and
$\phi_k\equiv 1$ in a neighborhood of the support of $\varphi_k.$ Calculations
that are similar to ones in \eqref{ft-est} show that
\begin{equation} \label{ft-west}
\|\phi_ku\|^2\leq \varepsilon' \tilde C_k(\| \overline{\partial}u
\|^2+\|\overline{\partial}^* u\|^2+\|u\|^2)+\tilde
C_{k,\varepsilon'}\|u\|^2_{-1}
\end{equation}
By choosing $\varepsilon ,\varepsilon' >0$ small enough and combining 
\eqref{ft-est} and \eqref{ft-west}  we get the following estimate: for all
$k=1,2,3,\cdots$ there exists $M_k>0$ such that 
\begin{equation} \label{pre-est}
\|u\|^2\leq \frac{2}{k}
(\|\overline{\partial}u\|^2+\|\overline{\partial}^*u\|^2+\|u\|^2)+M_k\|u\|_{-1}
^2 \text{ for } u\in Dom(\overline{\partial}^*)\cap
C_{(p,q)}^{\infty}(\overline{\D})
\end{equation}
We note that $Dom(\overline{\partial}^*)\cap C_{(p,q)}^{\infty}(\overline{\D})$
is dense in $Dom(\overline{\partial}^*)\cap Dom(\overline{\partial}) .$
Therefore, the above estimate \eqref{pre-est} holds on
$Dom(\overline{\partial}^*)\cap Dom(\overline{\partial}) .$ That is, the
$\overline{\partial}$-Neumann operator of $\D$ is compact on $(p,q)$-forms for
$0\leq p\leq n$ and $1\leq q \leq n-1.$ 
\end{proof}

\section*{Acknowledgement}
The author would like to thank his advisor, Emil Straube, for  suggesting the
problem and fruitful discussions, and Mehmet \c{C}elik for valuable comments on
a preliminary version of this manuscript.

\end{document}